\documentclass[12pt]{amsart}
\usepackage[english]{babel}
\usepackage[toc,page]{appendix}
\usepackage{amsmath}
\usepackage{amsthm}
\usepackage{amssymb}
\usepackage{mathrsfs}
\usepackage{enumerate}
\usepackage[notcite, final, notref]{showkeys}
\usepackage{dsfont}
\usepackage{bookmark}
\usepackage{hhline}
\usepackage[dvips]{color}
\usepackage{soul}
\usepackage{tikz}
\usetikzlibrary{decorations.markings}
\tikzset{>=latex}

\newcommand{\suchthat}{\;\ifnum\currentgrouptype=16 \middle\fi|\;}

\newtheorem{theorem}{Theorem}[section]

\newtheorem{lemma}[theorem]{Lemma}
\newtheorem{proposition}[theorem]{Proposition}
\newtheorem{corollary}[theorem]{Corollary}
\newtheorem{definition}[theorem]{Definition}

\theoremstyle{definition}
\newtheorem{remark}[theorem]{Remark}

\allowdisplaybreaks

\newcommand{\as}{{\mathbf A}}
\newcommand{\bbs}{{\mathbf B}}

\usepackage{tikz}
\usepackage[siunitx]{circuitikz}

\usepackage{xcolor}

\title[Extension of positive definite kernels]{On the extension of positive definite kernels to topological algebras}

\author[D. Alpay]{Daniel Alpay}
\address{(DA) 
Schmid College of Science and Technology\\
Chapman University\\
One University Drive\\
Orange, California 92866\\
USA}
\email{alpay@chapman.edu}

\author[I. L. Paiva]{Ismael L. Paiva }
\address{(ILP)
Schmid College of Science and Technology\\
Chapman University\\
One University Drive\\
Orange, California 92866\\
USA}
\email{depaiva@chapman.edu}

\thanks{Daniel Alpay thanks the Foster G. and Mary McGraw Professorship in
Mathematical Sciences, which supported this research. Ismael L. Paiva acknowledges financial support from the Science without Borders Program (CNPq/Brazil, Fund No. 234347/2014-7).}

\begin{document}

\tikzset{->-/.style={decoration={
  markings,
  mark=at position #1 with {\arrow[scale=2]{>}}},postaction={decorate}}}

\begin{abstract}
We define an extension of operator-valued positive definite functions from the real or complex setting to topological algebras, and describe their associated reproducing kernel spaces. The case of entire functions is of special interest, and we give a precise meaning to some power series expansions of analytic functions that appears in many algebras.
\end{abstract}
\maketitle


\date{today}
\setcounter{tocdepth}{1}
\tableofcontents


\section{Introduction}
\setcounter{equation}{0}

It is often of interest in some areas of mathematics to consider the extension of the domain (and, in general, the range) of analytic functions from the field of complex or real numbers, here denoted by $\mathbb{K}$, to an algebra $\mathcal A$ over $\mathbb{K}$. More precisely, if $f$ is analytic in a neighborhood $\Omega_\mathbb{K}$ of $z$, with Taylor expansion
\begin{equation}
f(z+h)=\sum_{n=0}^\infty h^n\frac{f^{(n)}(z)}{n!},
\label{power-series}
\end{equation}
where $h$ is a complex number in the open disk of convergence centered at $z$, one formally defines
\begin{equation}
\label{ext}
f(z+A)=\sum_{n=0}^\infty A^n\frac{f^{(n)}(z)}{n!},
\end{equation}
where $A\in\mathcal{A}\setminus\mathbb{K}$. Note that the restriction of $A\in\mathcal A\setminus\mathbb{K}$ is made just to keep Eq. \eqref{ext} familiar. In fact, one could consider the change $A\mapsto A+h$, where $h\in\mathbb K$ is such that $z+h\in\Omega_\mathbb{K}$.\smallskip

This type of extension is used many areas, like in supermathematics, and even in the theory of linear stochastic systems \cite{MR2610579,alp} and in the associated theory of strong algebras \cite{MR3029153,MR3404695}. Moreover, extensions given by Eq. \eqref{ext} contrasts with the one done, for instance, in the study of white noise space, where, with exception of \cite{aa_goh}, the complex coefficients of the power series (and not the variable) in Eq. \eqref{power-series} are replaced by elements that take value in the space of stochastic distributions $S_{-1}$.\smallskip

In some cases, if the functions of interest in Eq. \eqref{power-series} form a reproducing kernel Hilbert space, the extension of their kernel can be straightforward. Let, for instance, $\mathcal A$ be the Grassmann algebra with a finite number, say $N$, of generators.
In this case, a number $z\in\mathcal{A}$ can be written as
\[
z = z_B+z_S,
\]
where $z_B\in\mathbb{C}$ is called the body of $z$, and $z_S\in\mathcal{A}\setminus\mathbb{C}$ is such that $z_S^{N+1}=0$ and is called the soul of $z$. Then, one has
\begin{equation}
\mathscr{K}_N(z_B+z_S,w_B+w_S)=\begin{pmatrix}1&z_S&\frac{z_s^2}{2!}\cdots& \frac{z_S^N}{N!}\end{pmatrix}K_N(z_B,w_B)\begin{pmatrix}1\\
\overline{w_S}\\ \frac{\overline{w_S}^2}{2!}\\ \vdots \\ \frac{\overline{w_S}^N}{N!}\end{pmatrix},
\label{extkern}
\end{equation}
where $K_N$ is the $(N+1)\times (N+1)$ matrix function with $(n,m)$ entry equals to
\begin{equation}
\label{sofsof1}
\frac{1}{n!m!}\frac{\partial^{n+m} K(z,w)}{\partial^n \partial^m z^n\overline{w}^m}.
\end{equation}
Eq. \eqref{extkern} can be written in this simple form because of the nilpotence of the soul of $z$, i.e., because $z_S^{N+1}=0$.\smallskip

In general, for a Grassmann algebra with an infinite number of generators or for other arbitrary algebras, an expression similar to Eq. \eqref{extkern} is desirable. However, this requires a more careful analysis --- for instance, with the study of convergence.\smallskip

The present work focus on functions given by Eq. \eqref{ext} whenever $\mathcal A$ is a topological algebra, i.e., $\mathcal A$ is a locally convex topological vector space and the product $ab$ is separately continuous in each of the variables --- see Ref. \cite{MR0438123} for more details. Because of that, the convergence of Eq. \eqref{ext} is assumed to be in the topology of $\mathcal A$.\smallskip

Our objective is, in case the original analytic functions $f$ presented in Eq. \eqref{power-series} form a reproducing kernel Hilbert space, to introduce the structure of the corresponding space of extended functions given by Eq. \eqref{ext} and the extension of the underlying operators in this scenario.\smallskip

In a sense, our approach can be seeing as a reduction to the complex (or real) case. The reason is that, even though they are desirable, expressions like Eq. \eqref{extkern} do not seem to always exist. In general, we do not obtain a closed form for the kernel of functions $f$ presented in Eq. \eqref{ext}. Then, we study objects of the following type instead
\begin{equation}
\label{weakconv}
F(a,z,A) = \langle a, f(z+A)\rangle=\sum_{n=0}^\infty\langle a,(z+A)^n\rangle \frac{f^{(n)}(z)}{n!},
\end{equation}
where $a$ belongs to the topological dual $\mathcal{A}^\prime$. In other words, we replace the powers $(z+A)^n$ by $\langle a, (z+A)^n\rangle$ for every $n$. The exchange of order between the sum and the duality operation is justified because convergence in the topological algebra implies weak convergence.\smallskip

To introduce the underlying reproducing kernel Hilbert space, more definitions are required. However, we already remark that, for the Fock space, with reproducing kernel given by
\begin{equation}
\label{touttavie}
e^{z\overline w}=\sum_{n=0}^\infty \frac{z^n{\overline w}^n}{n!},
\end{equation}
the reproducing kernel of the space of extended functions is
\begin{equation}
\label{fock123}
\sum_{n=0}^\infty \frac{\langle a, (z+A)^n\rangle\overline{\langle b, (w+B)^n\rangle}}{n!}.
\end{equation}

In the case of $\mathbb{K}=\mathbb{C}$ and $\mathcal A=\mathbb C^{n\times n}$, the kernel in Eq. \eqref{fock123} becomes
\begin{equation}
K(a,A,z,b,B,w)={\rm Tr}\,\left((a^*\otimes b)e^{(zI_n+A)\otimes(\overline{w}I_n+ B^*)}\right),\quad a,A,b,B\in\mathbb C^{n\times n}.
\end{equation}

For yet another example, we take $\mathcal A$ to be the quaternions $\mathbb H$ and $\mathbb K$ to be the real numbers. Then, Eq. \eqref{fock123} becomes
\begin{equation}
\label{quaternionic-kernel}
\sum_{n=0}^\infty\frac{\left({\rm Re}\,\overline{a}(t+p)^n\right)\left({\rm Re}\,(\overline{q}+s)^nb\right)}{n!}.
\end{equation}
Even though the above expression is similar, it is not equivalent to the reproducing kernel of the Fock space of slice hyperholomorphic functions \cite{MR3587897}, which is
\[
\sum_{n=0}^\infty\frac{(t+p)^n(\overline{q}+s)^n}{n!}.
\]
Although the real part is taken in Eq. \eqref{quaternionic-kernel}, observe that non-real parts of the variables also play a role in the kernel. This can be seeing by varying the parameters $a$ and $b$.\smallskip

Besides this introduction, this paper contains four sections. In Section \ref{2}, we formalize our definition of reproducing kernel Hilbert spaces associated to extensions of functions of the type given by Eq. \eqref{ext}. After that, the case of entire functions is considered in Section 3. Then, the definitions and results are generalized to arbitrary analytic functions in Section 4. Finally, extensions of operators are considered in the last section.

\section{Extension of kernels to topological algebras}
\setcounter{equation}{0}
\label{2}

In this work, we assume that the algebra and its topological dual are endowed with involutions (for simplicity denoted by the same symbol) $A\mapsto A^*$ and $a\mapsto a^*$, which extend the complex conjugation, keep the algebraic structure, and satisfy
\begin{equation}
\overline{\langle a, A\rangle}=\langle a^*, A^*\rangle,\quad A\in\mathcal A\quad{\rm and}\quad a\in\mathcal A^\prime.
\label{azerty}
\end{equation}
In particular, choosing $A= cI$, where $c\in\mathbb{K}$, we have
\begin{equation}
\langle a,A^*\rangle=\overline{c}\cdot\langle a, I\rangle. 
\label{arizona}
\end{equation}

Before focusing on the case of entire functions, consider the case of a positive definite $\mathbf{B}(\ell_2(\mathbb{N}_0))$-valued kernel and $\Omega_\mathbb{K}\subset\mathbb{K}$. By definition of an operator-valued positive definite function, the $\mathbb{K}$-valued function
\begin{equation}
\mathcal{K}((z,f),(w,e))=\langle K(z,w)e,f\rangle_{\ell_2(\mathbb N_0)}
\label{alaska}
\end{equation} 
is positive definite on $\Omega_\mathbb{K}\times\ell_2(\mathbb N_0)$. Note that Eq. \eqref{alaska} can be rewritten as
\begin{equation}
\label{nebraska}
\mathcal{K}((z,f),(w,e))=\sum_{n,m=0}^\infty \overline{e_n}k_{nm}(z,w)f_m,
\end{equation}
where $(k_{nm}(z,w))_{n,m=0}^\infty$ is the matrix representation of $K(z,w)$ with respect to the standard basis of 
$\ell_2(\mathbb N_0)$, and where the elements of $\ell_2(\mathbb N_0)$ are written as semi-infinite column vectors.\smallskip

We, then, extend $\mathcal K$ --- or, more precisely, Eq. \eqref{nebraska} --- to the domain $\Omega$ in the following way:
\begin{equation}
\label{missouri}
\begin{split}  
\mathcal {K}((z,(A_n)_{n=0}^\infty,a),(w,(B_m)_{n=0}^\infty,b))&=\sum_{n,m=0}^\infty \overline{\langle a,A_n^*\rangle} k_{nm}(z,w)\langle b, B_m^*\rangle\\
&=\sum_{n,m=0}^\infty \langle a^*,A_n\rangle k_{nm}(z,w)\overline{\langle b^*, B_m\rangle}.
\end{split}
\end{equation}
In this expression, $a$ and $b$ belong to the topological dual $\mathcal{A}^\prime$ of the algebra $\mathcal{A}$ and the brackets denote the duality between $\mathcal{A}$ and $\mathcal{A}^\prime$. Moreover, $(A_n)_{n=0}^\infty$ and $(B_m)_{m=0}^\infty$ are sequences of elements indexed by $\mathbb N_0$. If the entries $k_{nm}$ are matrix-valued, say in $\mathbb C^{p\times p}$, we take the $A_n$ and the $B_m$ to be in $\mathcal{A}^{1\times p}$. The duality expressions $\langle a, A_n^*\rangle$ and $\langle b, B_m^*\rangle$ are, then, in $\mathbb C^{1\times p}$.\smallskip

The function given by Eq. \eqref{missouri} is well-defined and positive definite on the set $\Omega\equiv\Omega_\mathbb{K}\times\Omega_\mathcal{A}$, with
\[
\Omega_\mathcal{A}=\left\{(a,(A_n^*)_{n=0}^\infty)\,\,\mbox{{\rm such that}}\,\, (\langle a,A_n^*\rangle)_{n=0}^\infty\,\in\ell_2(\mathbb N_0)\right\}
\]
and is the starting point of our study.\smallskip

\begin{remark}{\rm In principle, one can replace in Eq. \eqref{nebraska} the $\mathbb{K}$ numbers $e_n$ and $f_m$ by elements in this algebra, and the complex conjugation by the conjugation in $\mathcal A$ to get the expression
\begin{equation}
\label{alabama}
M((z,(A_n)_{n=0}^\infty),(w,(B_m)_{m=0}^\infty))=\sum_{n,m=0}^\infty A_nk_{nm}(z,w)B_m^{*},
\end{equation}
where the sequences $\as=(A_n)_{n=0}^\infty$ and $\bbs=(B_m)_{m=0}^\infty$ are chosen such that Eq. \eqref{alabama} converges. When $\mathcal A$ is an algebra of Hilbert operators, or a $C^*$-algebra, one can define positivity for Eq. \eqref{alabama}, but, in general, it is not so clear in which sense \eqref{alabama} defines a positive definite function, when convergent.}
\end{remark}

For general topological algebras, we define positivity using the $\mathbb K$-valued function in Eq. \eqref{missouri}. For $(a,(A_u))\in\Omega_{\mathcal A}$, we denote by $X(a,\as)$ the $\ell_2(\mathbb N_0)$ element with $u$-component $\langle a,A_u^*\rangle$
\begin{equation}
X(a,\as)=\begin{pmatrix}\langle a , A_0^*\rangle\\ \langle a , A_1^*\rangle\\ \langle a , A_2^*\rangle\\ \vdots\end{pmatrix}.
\end{equation}
Moreover, we let
\begin{equation}
\label{blabla}  
K(z,w)=\Gamma(z)\Gamma(w)^*
\end{equation}
be a minimal factorization of the $\mathbf{B}(\ell_2(\mathbb{N}_0))$-valued kernel $K(z,w)$ via a Hilbert space $\mathcal G$, meaning that $\Gamma(z)\in\mathbf{B}(\mathcal G,\ell_2(\mathbb N_0))$ for every $z\in\Omega_\mathbb{K}$ and that the linear span of the range of the operators $\Gamma(w)^*$ is dense in $\mathcal G$, as $w$ runs through $\Omega_\mathbb{K}$. One can, for instance, choose for $\mathcal G$ the reproducing kernel Hilbert space $\mathcal H(K)$ of $\ell_2(\mathbb N_0)$-valued functions with reproducing kernel $K(z,w)$ and 
\[
(\Gamma(w))(f)=f(w),\quad f\in\mathcal H(K).
\]
Then,
\[
\Gamma^*(w)\xi=K(\cdot, w)\xi,\quad \xi\in \ell_2(\mathbb N_0)
\] 
and the next proposition is immediate:

\begin{proposition}
The factorization
\[
\begin{split}
\langle K(z,w)X(b,\bbs),X(a,\as)\rangle_{\ell_2(\mathbb N_0)}&=\langle \Gamma(w)^*X(b,\bbs), \Gamma(z)^*X(a,\as)\rangle_{\mathcal G},
\end{split}
\]
holds and the reproducing kernel Hilbert space associated to Eq. \eqref{missouri} consists of functions of the form
\begin{equation}
\label{FzAa}
F(z,\as,a)=\langle f,\Gamma(z)^*X(a,\as)\rangle_{\mathcal G},\quad f\in\mathcal G,
\end{equation}
with inner product and norm induced from the inner product and the norm of $\mathcal G$.
\end{proposition}

We, now, introduce the matrix representation
\begin{equation}
\label{missouri2}
\Gamma(z)f=\begin{pmatrix}f_0(z)\\ f_1(z)\\ \vdots\end{pmatrix}
\end{equation}
of $\Gamma(z)$, and we associate to Eq. \eqref{FzAa} the $\mathcal A$-valued function
\begin{equation}
F(z,\as)=\sum_{n=0}^\infty A_nf_n(z).
\label{california}
\end{equation}
\begin{proposition}  
The series \eqref{california} is weakly convergent on the set of sequences $\mathbf A=(A_u)$ such that $X(a,\mathbf A)\in\ell_2(\mathbb N_0)$ for all $a\in\mathcal A^\prime$.
\end{proposition}

\begin{proof}
Starting from \eqref{FzAa}, we have:
\[
\begin{split}
\langle f,\Gamma(z)^*X(a,\as)\rangle_{\mathcal G}&=\langle \Gamma(z)f,X(a,\as)\rangle_{\ell_2(\mathbb N_0)}\\
&=\sum_{n=0}^\infty\langle a^*, A_n\rangle f_n(z)\\
&=\lim_{N\rightarrow\infty}\langle a^*,\sum_{n=0}^N A_nf_n(z)\rangle.
\end{split}
\]
\end{proof}

We are interested in the special case $f_n(z)=\dfrac{f^{(n)}(z)}{n!}$ and $A_n=(A^*)^n$. We, then, have the condition
\begin{equation}
\label{ell2ell2}
\sum_{n=0}^\infty\Big|\langle a^*,A^n\rangle\Big|^2<\infty
\end{equation}
to insure that $X(a,A)\in\ell_2(\mathbb N)$.\smallskip

In the next section, after introducing our approach for a general topological algebra, we explore two cases of special interest: strong algebras and Banach algebras.


\section{Analytic kernels for entire functions}
\setcounter{equation}{0}

Let $K(z,w)$ be a $\mathbb K^{p\times p}$-valued kernel, positive definite for $z,w\in\Omega_\mathbb{K}\subset\mathbb K$, and analytic in the variables $z$ and $\overline{w}$. Also, let $\mathcal H(K)$ denote the associated reproducing kernel Hilbert space with reproducing kernel $K$. Recall, see Ref. \cite{donoghue}, that the elements of $\mathcal H(K)$ are, then, analytic in $\Omega_\mathbb{K}$ and that, for every $n\in\mathbb N_0$, $w\in\Omega_\mathbb{K}$, and $\eta\in\mathbb K^p$, the function
\begin{equation}
D_{n,w}\eta\,:\,z\,\mapsto\,\frac{1}{n!}\frac{\partial^n K(z,w)\eta}{\partial\overline{w}^n}\,\,\in\,\,\mathcal H(K).
\end{equation}
Furthermore,
\begin{equation}
\langle f,D_{n,w}\eta\rangle_{\mathcal H(K)}=\frac{\eta^*f^{(n)}(w)}{n!},\quad\forall f\in\mathcal H(K).
\label{marseille}
\end{equation}
In particular,
\begin{equation}
\label{facto123321}
\langle D_{m,w}\xi\, ,\, D_{n,z}\eta\rangle_{\mathcal H(K)}=\frac{1}{m!n!}\eta^*\frac{\partial^{n+m}K(z,w)}{\partial z^n\partial\overline{w}^m}\xi=\eta^*\mathscr{K}_{n,m}(z,w)\xi,
\end{equation}
where $z,w\in \Omega_\mathbb{K}$ and $\xi,\eta\in\mathbb K^p$, and $\mathscr K_{n,m}(z,w)$ has been defined in Eq. \eqref{sofsof1}.\smallskip

We denote as vectors
\begin{equation}
\begin{pmatrix}
\xi_0 \\ \xi_1 \\ \xi_2\\ \vdots\end{pmatrix}
\end{equation}
the elements of  $\ell_2(\mathbb N_0,\mathbb K^p)$, i.e., the sequences of elements of $\mathbb K^p$ such that $\sum_{n=0}^\infty\|\xi_n\|^2<\infty$. Also, we let $f$ be a $\mathbb C^p$-valued function analytic in $\Omega_\mathbb{K}\subset\mathbb K$. Then,
\begin{equation}
\label{mathscrK}
J_z(f)\stackrel{\rm def.}{=}\begin{pmatrix}f(z)\vspace{1mm}
\\ f^{(1)}(z)\vspace{1mm}
\\ \frac{f^{(2)}(z)}{2!}\vspace{1mm}\\ \frac{f^{(3)}(z)}{3!}\\ \vdots \end{pmatrix},\quad z\in\Omega_\mathbb{K},
\end{equation}
is called the jet function generated by $f$ --- see Ref. \cite[p. 222]{MR2063356}. We denote by $J(f)$ the function
\begin{equation}
\label{Jf456}
z\,\mapsto\, J_z(f).
\end{equation} 

Now, we focus on the case of entire functions, i.e., $\Omega_\mathbb{K}=\mathbb K$.

\begin{lemma}
Let $K(z,w)$ be a $\mathbb K^{p\times p}$-valued positive definite kernel, entire in $z$ and $\overline{w}$, with associated reproducing kernel Hilbert space $\mathcal H(K)$. Also, let $f\in\mathcal H(K)$. Then, for all $z\in\mathbb K$, the operator $J_z\in\mathbf B(\mathcal H(K),\ell_2(\mathbb N_0,\mathbb K^p))$, and its adjoint is given by
\begin{equation}
J_z^*(u)=\sum_{n=0}^\infty \left(\frac{1}{n!}\frac{\partial^n K(\cdot,w)}{\partial\overline{w}^n}\Big|_{w=z}\right) u_n.
\end{equation}
\end{lemma}

\begin{proof}
The elements of $\mathcal H(K)$ are entire, and so for every $z\in\mathbb K$, the series
\[
f(z+1)=\sum_{n=0}^\infty 1^n\frac{f^{(n)}(z)}{n!}
\]
converges in norm, which implies that
\[
\sum_{n=0}^\infty\|\frac{f^{(n)}(z)}{n!}\|^2<\infty.
\]
The computation of $J_z^*$ goes as follows: with $u=(u_n)_{n=0}^\infty\in\ell_2(\mathbb N_0,\mathbb K^p)$ and $f\in\mathcal H(K)$,
\[
\begin{split}
\langle J_z(f),u\rangle_{\ell_2(\mathbb N_0)}&=\sum_{n=0}^\infty u_n^*\frac{f^{(n)}(z)}{n!}\\
&=\sum_{n=0}^\infty \langle f, D_{n,z}u_n\rangle_{\mathcal H(K)}\\
&=\langle f, J_z^*(u)\rangle_{\ell_2(\mathbb N_0)}.
\end{split}
\]

\end{proof}
\begin{theorem}
Let $K(z,w)$ be a $\mathbb K^{p\times p}$-valued function, entire in $z$ and $\overline{w}$.
Then:\smallskip
$(1)$ For every pair $(z,w)\in\mathbb K^2$, the semi-infinite block matrix
\begin{equation}
\mathscr{K}_{n,m}(z,w)=\frac{1}{m!n!}\frac{\partial^{n+m}K(z,w)}{\partial z^n\partial\overline{w}^m},\quad n,m=0,1,\ldots
\end{equation}
defines a bounded operator, which is denoted by $\mathscr{K}(z,w)$, from $\ell_2(\mathbb N_0,\mathbb K^p)$ into itself.\smallskip

$(2)$ The operator $\mathscr{K}(w,w)$ is Hermitian if the kernel $K(z,w)$ is Hermitian.\smallskip

$(3)$ The $\mathbf B(\ell_2,\ell_2)$-valued function $\mathscr{K}(z,w)$ is positive definite in $\mathbb{K}$ if the kernel $K(z,w)$ is positive definite in $\mathbb{K}$.

$(4)$ Assume the kernel $K(z,w)$ is positive definite in $\mathbb{K}$. Then,
\begin{equation}
\mathscr K(z,w)=J_zJ^*_w.
\end{equation}
\end{theorem}

\begin{proof} Let $(z,w)\in\mathbb K^2$. The power series expansion
\begin{equation}
K(z+M,w+M)=\sum_{n,m=0}^\infty \mathscr K_{n,m}(z,w)M^{n+m}
\end{equation}
converges for every $M>0$, and, in particular, there is a positive number $C=C(z,w,M)$ such that
\begin{equation}
|\mathscr{K}_{n,m}(z,w)M^{n+m}|\le C<\infty,\quad \forall n,m=0,1,\ldots
\end{equation}
Now, let $u=(u_n)_{n=0}^\infty$ and $v=(v_n)_{n=0}^\infty$ be two sequences in $\ell_2(\mathbb N_0,\mathbb K^p)$. Then, for $M>1$, and using the Cauchy-Schwartz inequality,
\[
\|\sum_{n=0}^\infty \frac{u_n}{M^n}\|\le\left(\sum_{n=0}^\infty\|u_n\|^2\right)^{1/2}\left(\sum_{n=0}^\infty M^{-2n}\right)^{1/2}=\frac{\|u\|}{\sqrt{1-\frac{1}{M^2}}}.
\]
Also, a similar result holds for $v$. Then, we have
\[
|u_n^*\mathscr{K}_{n,m}(z,w)v_m|\le\frac{\|u_n\|}{M^n}C(z,w,M)\frac{\|v_m\|}{M^m}.
\]
Finally,
\[
|\sum_{n,m=0}^\infty u_n^*\mathscr{K}_{n,m}(z,w)v_m|\le\frac{C}{1-\frac{1}{M^2}}\|u\|\cdot\|v\|.
\]
\end{proof}

As a consequence of the previous theorem, we can write:

\begin{corollary}
Let $\eta=(\eta_m)_{m=0}^\infty\in\ell_2(\mathbb N_0,\mathbb K^p)$ and $w\in\mathbb K$.
Then,
\begin{equation}
\sum_{m=0}^\infty D_{m,w}\eta_m \,\,\in\,\,\mathcal H(K)
\end{equation}
and
\begin{equation}
\label{groningen}
\mathscr{K}(z,w)\eta=\begin{pmatrix}(\sum_{m=0}^\infty D_{m,w}\eta_m)(z) \\ 
(\sum_{m=0}^\infty D_{m,w}\eta_m)^{(1)}(z)\\ \frac{1}{2!}(\sum_{m=0}^\infty D_{m,w}\eta_m)^{(2)}(z)\\ \vdots \end{pmatrix}.
\end{equation}
\end{corollary}

\begin{proof}
The first claim comes from 
\[
\langle \eta,\mathscr K(w,w)\eta\rangle_{\ell_2(\mathbb N_0)}=\|\sum_{m=0}^\infty D_{m,w}\eta_m\|_{\mathcal H(K)}^2.
\]
By Eq. \eqref{marseille},
\[
\frac{1}{n!}\left(\sum_{m=0}^\infty D_{m,w}\eta_m \right)^{(n)}=\langle \sum_{m=0}^\infty D_{m,w}\eta_m\, ,\, D_{n,z}\rangle_{\mathcal H(K)}.
\] 
However, this is the $n$-th entry of $\mathscr K(z,w)\eta$, as can be seen from Eq. \eqref{facto123321}.
\end{proof}

\begin{theorem}
Assume $K(z,w)$ is a $\mathbb K^{p\times p}$-valued positive definite function in $\mathbb K$ and entire in the variables $z$ and $\overline{w}$. The reproducing kernel Hilbert space associated to $\mathscr{K}(z,w)$ is the space of jet functions generated by the elements of $\mathcal H(K)$,
with inner product
\begin{equation}
\langle J(f),J(g)\rangle_{\mathcal H(\mathscr{K})}=\langle f,g\rangle_{\mathcal H(K)}.
\end{equation}
\end{theorem}

\begin{proof}
Using Eq. \eqref{groningen}, we have
\begin{equation}
\begin{split}
\langle J(f)(\cdot), \mathscr{K}(\cdot,w)\eta\rangle{\mathcal H(\mathscr{K})}&=\langle f,
\sum_{m=0}^\infty D_{m,w}\eta_m\rangle_{\mathcal H(K)}\\
&=\sum_{m=0}^\infty \eta_m^*\frac{f^{(m)}(w)}{m!}\\
&=\langle J(f)(w),\eta\rangle_{\ell_2(\mathbb N_0,\mathbb C^{p})}.
\end{split}
\end{equation}
\end{proof}

We, now, present two special cases.\\

{\bf Strong algebras:}

The notion of strong algebra was introduced in Refs. \cite{MR3029153,MR3404695}. It originated from a space of stochastic distributions defined by Y. Kondratiev\cite{MR1408433, MR592501} and a related inequality proved by V{\aa}ge \cite{MR2714906, vage96}.

\begin{definition}
An algebra $\mathcal A$ which is an inductive limit of a family of Banach spaces $\left\{\mathcal B_t\,:\, t\in T\right\}$ directed under inclusion is called a strong algebra if, for every $t\in T$, there exists $h(t)\in T$ such that, for every $s\ge h(t)$, there exists a positive constant $c_{s,t}$ such that, for every $A\in \mathcal B_t$ and $B\in\mathcal B_s$, the products $AB$ and $BA$ belong to $\mathcal B_s$ and
\begin{equation}
\|AB\|_s\le c_{s,t}\|A\|_t\cdot\|B\|_s\quad and\quad \|BA\|_s\le c_{s,t}\|A\|_t\cdot\|B\|_s.
\end{equation}
\end{definition}

Let $A\in\mathcal B_t\subset\mathcal A$, where $\mathcal A$ is a strong algebra. Also, let $s=h(t)$ and $d_t=c_{h(t),t}$. An easy induction shows that
\begin{equation}
\|A^n\|_{h(t)}\le d_t^{n-1}\|A\|_t^n,\quad n=1,2,\ldots
\end{equation}
Thus, for $a\in\mathcal A^\prime$, we have
\begin{equation}
\label{san-diego}
|\langle a, A^n\rangle|\le \|a\|^\prime\cdot d_t^{n-1}\|A\|_t^{n},\quad n=1,2,\ldots
\end{equation}
Hence, for $a\in \mathcal A^\prime$,
\begin{equation}
|\langle a,A^n\rangle|\le\frac{ \|a\|^\prime}{d_t}\left(c_t\|A\|_t\right)^n,\quad n=1,2,\ldots
\end{equation}

Therefore, the following theorem holds:
\begin{theorem}
In a strong algebra $\mathcal A$, the power series given by Eq. \eqref{ext}
\[
f(z+A)=\sum_{n=0}^\infty A^n\frac{f^{(n)}(z)}{n!}
\]
converges for every $A\in\mathcal A$. In particular, Eq. \eqref{ell2ell2} holds for all $a\in \mathcal A^\prime$ and all $A\in \mathcal A$.
\label{theorem}
\end{theorem}

\begin{proof}
This follows from Eq. \eqref{san-diego}.
\end{proof}

{\bf Banach algebras:}

The case of Banach algebras is much simpler than that of strong algebras. Indeed, when $\mathcal{A}$ is a Banach algebra, it and its dual are endowed with a norm, 
denoted by $\|\cdot\|$ and $\|\cdot\|^\prime$, respectively, and we have
\[
|\langle a,A^n\rangle|\le\|a\|^\prime\cdot\|A\|^n.
\]

Then, a version of Theorem \ref{theorem} also holds in this case.

\section{General analytic kernels}
\setcounter{equation}{0}

We, now, consider the case of kernels whose domain of analyticity in $z$ and $\overline{w}$ is not necessarily the entire set $\mathbb{K}$.

\begin{theorem}
Let $K(z,w)$ be a $\mathbb K^{p\times p}$-valued kernel analytic in $z$ and $\overline{w}$ in the open set $\Omega_{\mathbb{K}}$. Then, for every $(z,w)\in\Omega_{\mathbb{K}}^2$, there exists $M_0$ (which depends on $(z,w)$) such that, for every $M\in(0,M_0)$, the infinite matrix
\begin{equation}
\frac{1}{m!n!}\frac{\partial^{n+m}K(z,w)}{\partial z^n\partial\overline{w}^m}M^{n+m}
\end{equation}
defines a bounded operator from $\ell_2(\mathbb N_0)$ into itself.
\end{theorem}
\begin{proof}
Let $z,w\in\Omega_{\mathbb{K}}$. Then, there exists $M_0>0$ (which depends on $z$ and $w$) for which the power expansion
\[
K(z+M_0,w+M_0)=\sum_{n,m=0}^\infty \mathscr{K}_{n,m}(z,w)M_0^{n+m}
\]
converges. Let $C$ be such that
\[
|\mathscr{K}_{n,m}(z,w)M_0^{n+m}|\le C.
\]
For $M\in(0,M_0)$ we have
\[
\begin{split}
|\eta_n^*\mathscr{K}_{n,m}(z,w)M^{n+m}\xi_m|&=|\eta_n^*\frac{M_0^n}{M^n}\mathscr{K}_{n,m}(z,w)M_0^{n+m}\xi_m
\frac{\rho^m}{\rho_0^m}|\\
&\le C\left(\|a_n\|\frac{M^n}{M_0^n}\right)\left(\|\xi_m\|\frac{M^m}{M_0^m}|\right).
\end{split}
\]
Hence,
\[
|\sum_{n,m=0}^\infty\eta_n^*\mathscr{K}_{n,m}(z,w)M^{n+m}\xi_m|\le C \|\eta\|\cdot\|\xi\|\frac{1}{1-\frac{M}{M_0}}.
\]
\end{proof}

As a consequence, we have the following proposition, where
\begin{equation}
e(z)
\begin{pmatrix}I\\ \overline{z}I\\ \overline{z}^2I\\ \overline{z}^3I\\ \vdots\end{pmatrix},\quad z\in\mathbb K.
\label{extkern321}
\end{equation}

\begin{proposition}
Let $D(M)$ denotes the diagonal operator with diagonal equals to $(I,M,M^2,M^3,\cdots)$.
Thus,
\begin{equation}
\langle D(M)\mathscr K(z+h,w+k)D(M)\xi,\eta\rangle=\langle \mathscr K(z,w)e(k)\xi,e(h)\eta\rangle.
\end{equation}
\end{proposition}

The new positive definite kernel is, then,
\begin{equation}
\langle D(M)\mathscr K(z,w)D(M)X(b,B),X(a,A)\rangle_{\ell_2(\mathbb N_0)}.
\end{equation}

\section{Operators}
\setcounter{equation}{0}
Let $K$ be an analytic kernel, and $T$ in $\mathbf B(\mathcal H(K))$. Then, $T$ has a natural extension to an operator from $\mathcal H(\widetilde{K})$ into itself
via the formula
\begin{equation}
\widetilde{T}(J(f))=J(Tf)
\end{equation}

By definition of the norm in $\mathbf B(\mathcal H(K))$, the operator $\widetilde{T}$ is bounded if and only if $T$ is bounded. Furthermore we have the following result, the proof of which we omit.

\begin{proposition} Let $T$ and $S$ be possibly unbounded operators in $\mathcal H(K)$. Then,
\begin{eqnarray}
(\widetilde{TS})(J(f))&=&\widetilde{T}(\widetilde{S}J(f)),\quad f\in{\rm Dom}\, S\,\, such \,\, that,\,\, Sf\in{\rm Dom}\,T, \\
\widetilde{T^*}(J(f))&=&\widetilde{T}^*J(f),\quad f\in{\rm Dom}\, T.
\end{eqnarray}
\end{proposition}

In general, the case of unbounded operators is of interest, as the
Fock space example shows. Then, if
\begin{equation}
Z\equiv\begin{pmatrix}0&0&0&0&\cdots\\
1&0&0&0&\cdots\\
0&1&0&0&\cdots\\
\vdots&&&&\cdots
\end{pmatrix}
\end{equation}
and
\begin{equation}
S\equiv\begin{pmatrix}
0&1&0&0&0&\cdots\\
0&0&2&0&0&\cdots\\
0&0&0&3&0&\cdots\\
\vdots&&&&&\cdots
\end{pmatrix},
\end{equation}
we have
\begin{equation}
\widetilde{M_z}(J(f))=(zI+Z)J(f)
\end{equation}
and
\begin{equation}
\widetilde{\partial_z}(J(f))=SJ(f).
\end{equation}

Next, letting $h\in\mathcal G$ as in Eq. \eqref{blabla} and $\Gamma_n(z)h=\frac{f^{(n)}(z)}{n!}$, we extend the operators to the space of $\mathbb K$-valued functions with reproducing kernel given by Eq. \eqref{missouri} as
\begin{equation}
T_{\mathcal A}\left(\sum_{n=0}^\infty \langle a^*,A^n\rangle\Gamma_n(z)h\right)=\sum_{n=0}^\infty \langle a^*,A^n\rangle\Gamma_n(z)Th.
\end{equation}
When $\mathcal G=\mathcal H(K)$ in the factorization given by Eq. \eqref{blabla}, we have
\begin{equation}
T_{\mathcal A}\left(\sum_{n=0}^\infty \langle a^*,A^n\rangle\Gamma_n(z)h\right)=\sum_{n=0}^\infty \langle a^*,A^n\rangle\frac{(Th)^{(n)}(z)}{n!}
\end{equation}
or, equivalently,
\begin{equation}
\langle a^*, (Tf)(z+A)\rangle=\sum_{n=0}^\infty \langle a^*,A^n\rangle\frac{(Th)^{(n)}(z)}{n!}.
\end{equation}

\begin{proposition}
Let $T,S$ be two possibly unbounded linear operators from $\mathcal H(K)$ into itself.
Then,
\begin{eqnarray}
(TS)_{\mathcal A}&=&(T)_{\mathcal A}(S)_{\mathcal A},\\
(T_{\mathcal A})^*&=&(T^*)_{\mathcal A}.
\end{eqnarray}
\end{proposition}

Finally, we note that the operators $T,\widetilde{T}$ and $T_{\mathcal A}$ are related by
\begin{equation}
\begin{split}    
T_{\mathcal A}\left(\sum_{n=0}^\infty\langle a^*,A^n\rangle \Gamma_nh\right)=\langle \widetilde{T}J(f),X(a,(A^{n})_{n=0}^\infty)\rangle_{\ell_2(\mathbb N_0)}.
\end{split}  
\end{equation}
Indeed, we have
\[
\begin{split}
T_{\mathcal A}\left(\sum_{n=0}^\infty\langle a^*,A^n\rangle \Gamma_nh\right)&=\langle J(Tf),X(a,(A^{n})_{n=0}^\infty)\rangle_{\ell_2(\mathbb N_0)}\\
&=\langle \widetilde{T}(f),X(a,(A^{n})_{n=0}^\infty)\rangle_{\ell_2(\mathbb N_0)}.
\end{split}
\]


\bibliographystyle{plain}
\bibliography{all}

\def\cprime{$'$} \def\cprime{$'$} \def\cprime{$'$}
  \def\lfhook#1{\setbox0=\hbox{#1}{\ooalign{\hidewidth
  \lower1.5ex\hbox{'}\hidewidth\crcr\unhbox0}}} \def\cprime{$'$}
  \def\cprime{$'$} \def\cprime{$'$} \def\cprime{$'$} \def\cprime{$'$}
  \def\cprime{$'$}
\begin{thebibliography}{10}

\bibitem{aa_goh}
D.~Alpay and H.~Attia.
\newblock An interpolation problem for functions with values in a commutative
  ring.
\newblock In {\em {A Panorama of Modern Operator Theory and Related Topics}},
  volume 218 of {\em {Operator} {Theory}: {A}dvances and {A}pplications}, pages
  1--17. Birkh\"auser, 2012.

\bibitem{MR3587897}
D.~Alpay, F.~Colombo, I.~Sabadini, and G.~Salomon.
\newblock The {F}ock space in the slice hyperholomorphic setting.
\newblock In {\em Hypercomplex analysis: new perspectives and applications},
  Trends Math., pages 43--59. Birkh\"{a}user/Springer, Cham, 2014.

\bibitem{MR2610579}
D.~Alpay and D.~Levanony.
\newblock Linear stochastic systems: a white noise approach.
\newblock {\em Acta Appl. Math.}, 110(2):545--572, 2010.

\bibitem{alp}
D.~Alpay, D.~Levanony, and A.~Pinhas.
\newblock Linear stochastic state space theory in the white noise space
  setting.
\newblock {\em {SIAM} {Journal of Control and Optimization}}, 48:5009--5027,
  2010.

\bibitem{MR3029153}
D.~Alpay and G.~Salomon.
\newblock Topological convolution algebras.
\newblock {\em J. Funct. Anal.}, 264(9):2224--2244, 2013.

\bibitem{MR3404695}
D.~Alpay and G.~Salomon.
\newblock On algebras which are inductive limits of {B}anach spaces.
\newblock {\em Integral Equations Operator Theory}, 83(2):211--229, 2015.

\bibitem{MR2063356}
V.~Bolotnikov, A.~Kheifets, and L.~Rodman.
\newblock Jet functions having indefinite {C}arath\'eodory-{P}ick matrices.
\newblock {\em Linear Algebra Appl.}, 385:215--286, 2004.

\bibitem{donoghue}
W.F. Donoghue.
\newblock {\em Monotone matrix functions and analytic continuation}, volume 207
  of {\em Die {G}rundlehren der mathematischen {W}issennschaften}.
\newblock Springer--{V}erlag, 1974.

\bibitem{MR1408433}
H.~Holden, B.~{\O}ksendal, J.~Ub{\o}e, and T.~Zhang.
\newblock {\em Stochastic partial differential equations}.
\newblock Probability and its Applications. Birkh\"auser Boston Inc., Boston,
  MA, 1996.

\bibitem{MR592501}
Y.~Kondratiev.
\newblock Nuclear spaces of entire functions in problems of
  infinite-dimensional analysis.
\newblock {\em Dokl. Akad. Nauk SSSR}, 254(6):1325--1329, 1980.

\bibitem{MR0438123}
M.~A. Na{\u\i}mark.
\newblock {\em Normed algebras}.
\newblock Wolters-Noordhoff Publishing, Groningen, third edition, 1972.
\newblock Translated from the second Russian edition by Leo F. Boron,
  Wolters-Noordhoff Series of Monographs and Textbooks on Pure and Applied
  Mathematics.

\bibitem{MR2714906}
G.~V{\aa}ge.
\newblock {\em Stochastic differential equations and {K}ondratiev spaces}.
\newblock ProQuest LLC, Ann Arbor, MI, 1995.
\newblock Thesis (Dr.Ing.)--Universitetet i Trondheim Norges Tekniske Hogskole
  (Norway).

\bibitem{vage96}
G.~V{\aa}ge.
\newblock Hilbert space methods applied to stochastic partial differential
  equations.
\newblock In H.~K{\"o}rezlioglu, B.~{\O}ksendal, and A.S. {\"U}st{\"u}nel,
  editors, {\em Stochastic analysis and related topics}, pages 281--294.
  Birk{\"a}user, {B}oston, 1996.

\end{thebibliography}
\end{document}